\documentclass[reqno,10pt]{amsart}

\usepackage{xypic}
\input xy
\xyoption{all}
\usepackage{epsfig}
\usepackage{color}
\usepackage{amsthm}
\usepackage{amssymb}
\usepackage{amsmath}
\usepackage{amscd}
\usepackage{amsopn}
\usepackage{url}

\newcommand{\labell}[1] {\label{#1}}

\numberwithin{equation}{section}
\newtheorem {Theorem}{Theorem}
\numberwithin{Theorem}{section}

\newtheorem {Lemma}[Theorem]{Lemma}

\newtheorem {Proposition}[Theorem]{Proposition}

\theoremstyle{definition}
\newtheorem{Definition}[Theorem]{Definition}
\theoremstyle{remark}
\newtheorem{Remark}[Theorem]{Remark}

\expandafter\chardef\csname pre amssym.def at\endcsname=\the\catcode`\@
\catcode`\@=11
\def\undefine#1{\let#1\undefined}
\def\newsymbol#1#2#3#4#5{\let\next@\relax
 \ifnum#2=\@ne\let\next@\msafam@\else
 \ifnum#2=\tw@\let\next@\msbfam@\fi\fi
 \mathchardef#1="#3\next@#4#5}
\def\mathhexbox@#1#2#3{\relax
 \ifmmode\mathpalette{}{\m@th\mathchar"#1#2#3}%
 \else\leavevmode\hbox{$\m@th\mathchar"#1#2#3$}\fi}
\def\hexnumber@#1{\ifcase#1 0\or 1\or 2\or 3\or 4\or 5\or 6\or 7\or 8\or
 9\or A\or B\or C\or D\or E\or F\fi}

\font\teneufm=eufm10
\font\seveneufm=eufm7
\font\fiveeufm=eufm5
\newfam\eufmfam
\textfont\eufmfam=\teneufm
\scriptfont\eufmfam=\seveneufm
\scriptscriptfont\eufmfam=\fiveeufm

\catcode`\@=\csname pre amssym.def at\endcsname

\newcommand{\ai}{\mathfrak{a}}

\newcommand{\FF}{{\mathcal F}}

\newcommand{\PP}{{\mathcal P}}

\newcommand{\SC}{{\mathcal S}}

\def    \R      {{\mathbb R}}
\def    \Z      {{\mathbb Z}}

\newcommand{\supp}{{\mathit supp}\,}

\newcommand{\id}{{\mathit id}}

\def    \p        {\partial}

\def    \Fix      {\operatorname{Fix}}

\def \codim {\operatorname{codim}} 

\def    \Hamt {{ \widetilde{ \operatorname{Ham} } }} 
\def    \HF   {\operatorname{HF}} 

\def \wM {\omega_{\scriptscriptstyle{M}}}

\begin{document}


\setlength{\smallskipamount}{6pt}
\setlength{\medskipamount}{10pt}
\setlength{\bigskipamount}{16pt}




\title[leafwise Coisotropic Intersections]{leafwise Coisotropic
Intersections}

\author[Ba\c sak G\"urel]{Ba\c sak Z. G\"urel}

\address{Department of Mathematics, Vanderbilt University, Nashville,
  TN 37240, USA} \email{basak.gurel@vanderbilt.edu}

\date{\today} 



\begin{abstract}
  We establish the leafwise intersection property for closed,
  coisotropic submanifolds in an exact symplectic manifold satisfying
  natural additional assumptions.
\end{abstract}

\maketitle


\section{Introduction and main results }
\labell{sec:main-results}
In this paper, we study the question of existence of leafwise
intersections for coisotropic submanifolds -- one of the
generalizations of the well-known Lagrangian intersection property to
coisotropic submanifolds. More specifically, we establish the
existence of leafwise intersections for restricted contact type
coisotropic submanifolds and Hamiltonian diffeomorphisms with Hofer
energy below a certain natural, sharp threshold. In other words, under
these conditions, we prove the existence of a leaf $F$ of the
characteristic foliation on a coisotropic submanifold $M$ such that
$\varphi(F) \cap F \neq \emptyset$ for a Hamiltonian diffeomorphism
$\varphi$.

The problem of existence of coisotropic intersections was first
addressed by Moser in the late 70s, \cite{Mo78}, and since then
various forms of the coisotropic intersection property have been
established.  For instance, the existence of ordinary intersections of
$M$ and $\varphi(M)$, or, roughly equivalently, lower bounds on the
displacement energy of $M$ have also been investigated. In this vein,
positivity of the displacement energy for stable coisotropic
submanifolds of $\R^{2n}$ was proved by Bolle, \cite{Bo96,Bo98}. (See
Section \ref{subsec:prelim-lw} for definitions.) In \cite{Gi07coiso},
revitalizing the subject a decade after Bolle's work, Ginzburg
extended Bolle's result to general symplectically aspherical
manifolds. This was generalized to closed, rational symplectic
manifolds by Kerman, \cite{Ke07coiso}, and, more recently, Usher,
\cite{U09coiso}, proved that the displacement energy of a stable
coisotropic submanifold of any closed or convex symplectic manifold is
positive.  Currently, the question is also being studied by Tonnelier,
\cite{To}.

The main result of this paper, falling in the realm of leafwise
intersections, is

\begin{Theorem}
\label{thm:lw}
Let $(W^{2n},\,\omega)$ be an exact symplectic manifold with
$c_1|_{\pi_2(W)} = 0$ which is geometrically bounded and wide. Let $M
\subset W$ be a closed, coisotropic submanifold of restricted contact
type. Then, for any compactly supported Hamiltonian diffeomorphism
$\varphi_H \colon W \to W$ with $\|H\| < \ai(M)$, there exists a leaf
$F$ of the characteristic foliation $\FF$ on $M$ such that
$\varphi_H(F) \cap F \neq \emptyset$. Here,
$$
\ai(M) = \inf \{
A(\gamma) >0 \mid \gamma \text{ is a loop that is tangent to } \FF
\text{ and contractible in } W \},
$$ 
where $A(\gamma)$ is the (negative) symplectic area bounded by
$\gamma$.
\end{Theorem}

\begin{Remark} 
  Note that $\ai(M) >0$ whenever $M$ has restricted contact type; see
  \cite{Gi07coiso}. Furthermore, we set $\ai(M) = \infty$ if there is
  no $\gamma$ as in the definition of $\ai(M)$.
\end{Remark}

\begin{Remark}
  The assumption that $c_1|_{\pi_2(W)} = 0$ is not really essential in
  Theorem \ref{thm:lw} and included only for the sake of
  completeness. To be more precise, several of the results used in the
  proof of the theorem were originally established for symplectically
  aspherical manifolds. However, it is clear these results also hold
  for exact, wide, geometrically bounded manifolds without any
  additional assumptions on the first Chern class.
\end{Remark}

A symplectic manifold $W$ is said to be wide if it is open and admits
an arbitrarily large compactly supported Hamiltonian without
non-trivial contractible fast periodic orbits.  This notion was
introduced and discussed in detail in \cite{Gu07}. In particular,
examples of wide manifolds include manifolds that are convex at
infinity (e.g., $\R^{2n}$, cotangent bundles, Stein manifolds) and
twisted cotangent bundles. The essence of the wideness property lies
in the fact that on a wide manifold the top degree Floer homology is
non-zero for any non-negative compactly supported Hamiltonian which is
not identically zero. This allows one to construct an action selector
for geometrically bounded and wide manifolds, a tool utilized in the
proof of Theorem \ref{thm:lw}. (It is not known how to define action
selectors for an arbitrary geometrically bounded manifold.)

As has been pointed out above, the problem of leafwise intersections
was first considered in \cite{Mo78} where Moser established the
existence of leafwise intersections whenever $\varphi$ is $C^1$-close
to the identity and $M$ is simply connected.  The latter assumption
was removed by Banyaga, \cite{Ba}. Furthermore, Ekelund and Hofer,
\cite{EH,Ho90a}, showed that for contact type hypersurfaces in
$\R^{2n}$, the $C^1$-smallness assumption on $\varphi$ can be replaced
by a much less restrictive hypothesis that the Hofer norm of $\varphi$
is smaller than a certain symplectic capacity of the domain bounded by
$M$. More recently, Ginzburg, \cite{Gi07coiso}, proved the existence
of leafwise intersections for restricted contact type hypersurfaces of
sub-critical Stein manifolds. The question has also been studied by
Ziltener, \cite{Zilt08}, under the additional assumption that the
characteristic foliation is a fibration.

Perhaps the most relevant to this work are recent papers
\cite{AF08,AF08generic} by Albers and Frauenfelder. In \cite{AF08},
the authors establish the existence of leafwise intersections for
hypersurfaces $M$ of restricted contact type, provided that the
ambient symplectic manifold is exact and convex at infinity and that
the Hofer norm of $\varphi$ does not exceed the minimal Reeb period
$\ai(M)$. In a follow-up work, \cite{AF08generic}, they prove that,
generically, the number of leafwise intersection points is at least
the sum of $\Z_2$-Betti numbers of $M$. The method used in
\cite{AF08,AF08generic} is the Rabinowitz Floer homology (see, e.g.,
\cite {CFO}). The result of the present paper is a generalization of
the theorem in \cite{AF08} to the case where $\codim M > 1$, without
the multiplicity lower bound on the number of leafwise
intersections. Thus, the leafwise intersection property holds not just
for hypersurfaces, but is a feature of coisotropic submanifolds in
general.

\begin{Remark}
  Comparing Theorem \ref{thm:lw} with the result from \cite{Mo78,Ba},
  note that the latter holds for any coisotropic submanifold $M$, but
  the requirement on $\varphi$ is very restrictive. On the other hand,
  in Theorem \ref{thm:lw} the roles are reversed: we impose a strong
  condition on $M$ while the requirement on $\varphi$ has been
  relaxed. It is then natural ask whether Theorem \ref{thm:lw} can be
  generalized to any coisotropic submanifold $M$ along the lines of
  \cite{Mo78,Ba}. This, however, is not the case and leafwise
  intersections appear to be fragile:\emph{ There exists a
    hypersurface $M \subset \R^{2n}$ (diffeomorphic to
    $S^{2n-1}$) and a sequence of (autonomous) Hamiltonians $H_i
    \colon \R^{2n} \to \R$, supported within the same compact set,
    such that
\begin{itemize}
\item $\parallel H_i \parallel _{C^0} \to 0$, and
\item $\varphi_{H_i}(M)$ and $M$ have no leafwise intersections.
\end{itemize}} 
\noindent This result will be proved in the forthcoming paper
\cite{Gu09cex}.
\end{Remark}

\subsection*{Acknowledgments} 
The author is deeply grateful to Viktor Ginzburg for many useful
discussions and his numerous valuable remarks and suggestions.


\section{Preliminaries}
\label{sec:lw-int}
We start this section by recalling the relevant definitions and basic
results concerning coisotropic submanifolds.  In Section
\ref{subsec:conv}, we set our conventions and notation.  We recall the
definition of the action selector for wide manifolds and state its
relevant properties in Section \ref{subsec:selector}.

\subsection{Preliminaries on coisotropic submanifolds}
\label{subsec:prelim-lw}
Let $(W^{2n}, \omega)$ be a symplectic manifold and let $M\subset W$
be a closed, coisotropic submanifold of codimension $k$. Set
$\wM=\omega|_M$. Then, as is well known, the distribution $\ker \wM$
has dimension $k$ and is integrable. Denote by $\FF$ the
characteristic foliation on $M$, i.e., the $k$-dimensional foliation
whose leaves are tangent to the distribution $\ker \wM$.

\begin{Definition}
\label{def:B}
The coisotropic submanifold $M$ is said to be \emph{stable} if there
exist one-forms $\alpha_1,\ldots,\alpha_k$ on $M$ such that $\ker
d\alpha_i\supset \ker\wM$ for all $i=1,\ldots,k$ and
\begin{equation*}
\labell{eq:ct}
\alpha_1\wedge\cdots\wedge\alpha_k\wedge \wM^{n-k}\neq 0
\end{equation*}
anywhere on $M$. We say that $M$ has \emph{contact type} if the forms
$\alpha_i$ can be taken to be primitives of $\wM$. Furthermore, $M$
has \emph{restricted} contact type if the forms $\alpha_i$ extend to
global primitives $\bar{\alpha}_i$ of $\omega$ on $W$.
\end{Definition}

Stable and contact type coisotropic submanifolds were introduced by
Bolle in \cite{Bo96,Bo98} and considered in a more general setting by
Ginzburg, \cite{Gi07coiso}, and afterwards also by Kerman,
\cite{Ke07coiso}, and by Usher, \cite{U09coiso}. (See also \cite{Dr}.)
Referring the reader to \cite{Gi07coiso} for a discussion of the
requirements of Definition \ref{def:B} and several illustrating
examples, let us merely note that these requirements are natural but
quite restrictive. For example, a stable Lagrangian submanifold is
necessarily a torus. In this paper, we will be mainly concerned with
coisotropic submanifolds of restricted contact type. This condition is
a generalization of its namesake for hypersurfaces, as is the case for
contact type or stability conditions.

Assume that $M$ is stable. The normal bundle to $M$ in $W$ is trivial
since it is isomorphic to $T^*\FF$ and, thus it can be identified with
$M \times \R^k$. From now on we identify a small neighborhood of $M$
in $W$ with a neighborhood of $M$ in $T^*\FF = M \times \R^k$ and use
the same symbols $\wM$ and $\alpha_i$ for differential forms on $M$
and for their pullbacks to $M \times \R^k$. (Thus, we will be
suppressing the pullback notation $\pi^*$, where $\pi \colon M \times
\R^k \to M$, unless its presence is absolutely necessary.)  Using the
Weinstein symplectic neighborhood theorem, we then have

\begin{Proposition}[\cite{Bo96, Bo98}]
\label{prop:normalform}
Let $M$ be a closed, stable coisotropic submanifold of $(W^{2n},
\omega)$ with $ \codim M = k$.  Then, for a sufficiently small $r>0$,
there exists a neighborhood of $M$ in $W$ which is symplectomorphic to
$U_r=\{ (q,p) \in M \times \R^k \mid |p| < r \}$ equipped with the
symplectic form $\omega = \wM + \sum_{j=1}^k d(p_j \alpha_j)$. Here
$(p_1,\ldots,p_k)$ are the coordinates on $\R^k$ and $|p|$ is the
Euclidean norm of $p$.
\end{Proposition}

As an immediate consequence of Proposition \ref{prop:normalform}, we
obtain a family of coisotropic submanifolds $M_p = M\times \{p\} $,
for $p \in B^k_r$, of $W$ which foliate a neighborhood of $M$ in $W$,
where $B_r^k$ is the ball of radius $r$, centered at the origin in
$\R^k$. Moreover, a leaf of the characteristic foliation on $M_p$
projects onto a leaf of the characteristic foliation on $M$.

Furthermore, we have

\begin{Proposition}[\cite{Bo96, Bo98, Gi07coiso}]
\label{prop:flow}
Let $M$ be a stable coisotropic submanifold.

\begin{enumerate}

\item[(i)] The leafwise metric $(\alpha_1)^2+\cdots+(\alpha_k)^2$ on
  $\FF$ is leafwise flat.

\item[(ii)] The Hamiltonian flow of
  $\rho=(p_1^2+\cdots+p_k^2)/2=|p|^2/2$ is the leafwise geodesic flow
  of this metric.

\end{enumerate}

\end{Proposition}

\begin{Remark}
  Using this property, it is not hard to show that, whenever $M$ has
  restricted contact type, for every leaf $F$ of $\FF$, the kernel of
  the map $ H_1(F;\R) \to H_1(W;\R)$ is either trivial or one
  dimensional. For instance, if $H_1(W;\R)=0$, every leaf $F$ is
  diffeomorphic to either $\R^k$ or $S^1 \times \R^{k-1}$.
\end{Remark}

\subsection{Conventions and notation} 
\label{subsec:conv}
In this section we set our conventions and notation. 

Let $(W^{2n},\omega)$ be a symplectically aspherical manifold, i.e.,
$\omega|_{\pi_2(W)}=c_1|_{\pi_2(W)}=0$.  Denote by $\Lambda W$ the
space of smooth contractible loops $\gamma\colon S^1\to W$ and
consider a time-dependent Hamiltonian $H\colon S^1\times W\to \R$,
where $S^1=\R/\Z$.  Setting $H_t = H(t,\cdot)$ for $t\in S^1$, we
define the action functional $A_H\colon \Lambda W\to \R$ by
$$
A_H(\gamma)=A(\gamma)+\int_{S^1} H_t(\gamma(t))\,dt,
$$
where $A(\gamma)$ is the negative symplectic area bounded by $\gamma$, i.e.,
$$
A(\gamma)=-\int_{\bar{\gamma}}\omega,
$$
where $\bar{\gamma}\colon D^2\to W$ is such that
$\bar{\gamma}|_{S^1}=\gamma$.

The least action principle asserts that the critical points of $A_H$
are exactly contractible one-periodic orbits of the time-dependent
Hamiltonian flow $\varphi_H^t$ of $H$, where the Hamiltonian vector
field $X_H$ of $H$ is defined by $i_{X_H}\omega=-dH$. We denote the
collection of such orbits by $\PP_H$. The action spectrum $\SC(H)$ of
$H$ is the set of critical values of $A_H$. In other words, $\SC(H)=\{
A_H(\gamma)\mid \gamma\in \PP_H\}$. This is a zero measure set; see,
e.g., \cite{HZ94}.

In what follows we will always assume that $H$ is compactly supported
and set $\supp H=\bigcup_{t\in S^1}\supp H_t$.  In this case, $\SC(H)$
is compact and hence nowhere dense.

Let $J=J_t$ be a time-dependent almost complex structure on $W$. A Floer
anti-gradient trajectory $u$ is a map $u\colon \R\times S^1\to W$
satisfying the equation
\begin{equation*}
\label{eq:floer}
\frac{\p u}{\p s}+ J_t(u) \frac{\p u}{\p t}=-\nabla H_t(u).
\end{equation*}
Here, the gradient is taken with respect to the time-dependent
Riemannian metric $\omega(\cdot,J_t\cdot)$. Denote by $u(s)$ the curve
$u(s,\cdot)\in \Lambda W$.

The energy of $u$ is defined as
\begin{equation*}
\label{eq:energy}
E(u)
=
\int_{-\infty}^\infty \left\|\frac{\p u}{\p s}\right\|_{L^2(S^1)}^2\,ds
=
\int_{-\infty}^\infty \int_{S^1}\left\|\frac{\p u}{\p t}-J\nabla H (u)
\right\|^2 \,dt\,ds.
\end{equation*}
We say that $u$ is asymptotic to $x^\pm\in \PP_H$ as $s\to\pm \infty$
or connecting $x^-$ and $x^+$ if $\lim_{s\to\pm\infty} u(s)=x^\pm$ in
$\Lambda W$. More generally, $u$ is said to be partially asymptotic to
$x^\pm\in \PP_H$ at $\pm\infty$ if $u(s^{\pm}_n)\to x^\pm$ for some
sequences $s^{\pm}_n\to\pm\infty$. In this case
$$
A_H(x^-)-A_H(x^+)=E(u).
$$

In this paper the manifold $W$ is assumed to be exact and, hence,
open. In this case, in order for the Floer homology to be defined, we
assume that $W$ is geometrically bounded. This assumption gives us
sufficient control of the geometry of $W$ at infinity which is
necessary in the case of open manifolds. Examples of such manifolds
include symplectic manifolds that are convex at infinity (e.g.,
$\R^{2n}$, cotangent bundles) as well as twisted cotangent
bundles. (See, e.g., \cite{AL,CGK} for the precise definition and a
discussion of geometrically bounded manifolds.) Under the hypotheses
that $W$ is symplectically aspherical and geometrically bounded, the
compactness theorem for Floer's connecting trajectories holds and the
filtered $\Z$-graded Floer homology of a compactly supported
Hamiltonian on $W$ is defined; see, e.g., \cite{CGK,GG04}.

\subsection{Action selector for wide manifolds}
\label{subsec:selector}
In this section we briefly recall the definition and relevant
properties of the action selector defined in \cite{Gu07} for wide and
geometrically bounded manifolds. We refer the reader to \cite{Gu07,
  Gi07coiso} for more details. Here we only note that action selectors
were constructed in \cite{schw00} for closed manifolds and in
\cite{HZ94} and \cite{Vi92} for $\R^{2n}$ and cotangent bundles,
respectively. The approach of \cite{schw00} has been extended to
manifolds convex at infinity in \cite{FS03}.

\emph{The definition.}  Let $H\colon S^1 \times W\to \R$ be a
compactly supported, non-negative Hamiltonian which is not identically
zero. As was proved in \cite{Gu07}, on a wide manifold, the top degree
Floer homology of $H$ for the action interval $(0,\infty)$ is non-zero
and it carries a canonically defined homology class. (Note that the
homology group itself depends on $H$.) Call this class $[\max H]$ and
define
$$
\sigma(\varphi_H)=\inf\{a> 0\mid j^a_H\left([\max H]\right)=0\}\in
  \SC(H),
$$
where
$$
j^a_H\colon \HF^{(0,\,\infty)}(H)\to \HF^{(a,\,\infty)}(H)
$$
is the quotient map. (This definition coincides with the one from
\cite{FS03} whenever $W$ is convex.)

\emph{Properties of the action selector.}  Focusing on the ones that
are relevant for what follows, recall that the action selector
$\sigma$, defined as above, has the following properties for
non-negative Hamiltonians:

\begin{enumerate}

\item[(S1)] $\sigma$ is monotone, i.e.,
  $\sigma(\varphi_K)\leq\sigma(\varphi_H)$, whenever $0\leq K\leq H$
  point-wise;

\item[(S2)] $0\leq\sigma(\varphi_H)\leq E^+(H)$ for any $H\geq 0$,
  where
$$
E^+(H)=\int_{S^1}\max_W H_t\,dt;
$$

\item[(S3)] $\sigma(\varphi_H)>0$, provided that $H\geq 0$ is not
  identically zero;

\item[(S4)] $\sigma(\varphi_H)$ is continuous in $H$ in the
  $C^0$-topology.

\end{enumerate}
We refer the reader to \cite{Gu07} for the proofs of these properties.


\section {Proof of Theorem \ref{thm:lw}}
\label{sec:lw-proof}
In this section we prove Theorem \ref{thm:lw}. Throughout the proof,
as in Section \ref{subsec:prelim-lw}, a neighborhood of $M$ in $W$ is
identified with a neighborhood of $M$ in $M\times \R^k$ equipped with
the symplectic form $\omega = \wM + \sum_{j=1}^k d(p_j \alpha_j)$.
Using this identification, we denote by $U_R$, with $R>0$ sufficiently
small, the neighborhood of $M$ in $W$ corresponding to $M\times
B^k_R$. (Thus, $U_R=\{\rho < R^2 /2\}$.)

\begin{proof}[Proof of Theorem \ref{thm:lw}]
  First note that, without loss of generality, we may assume that
  $\Fix(\varphi_H) \cap M = \emptyset$, for otherwise the assertion is
  obvious. Then $\varphi_H$ has no fixed points near $M$, say in a
  tubular neighborhood $U=M \times B_R^k$ of $M$ in $W$, where $R>0$
  is sufficiently small. Let $U_r \subset U$ be a smaller tubular
  neighborhood of $M$ for some $r<R$. We may also require that $H \geq
  0$ and $\| H\|= E^+(H)$. (This can be achieved by replacing $H$ by
  $f \cdot (H-\min H)$, where $f$ is a cut-off function equal to one
  near $\supp(H) \cup \bar{U}$.)

  Let $K\geq 0$ be a non-negative function on $[0,r]$ such that

\begin{itemize}
\item $K(0) = \max K >0$ and $K \equiv 0 $ near $r$;
\item $K$ is strictly decreasing until it becomes zero;
\item all odd-order derivatives of $K$ vanish at $0$, and $K''(0)<0$
is close to zero. 
\end{itemize}
Abusing notation, we also denote by $K$ the function on $W$, equal to
$K(|p|)$ on $U_r$, where $ |p| = \sqrt{2 \rho}$, and extended to be
identically zero outside $U_r$. Note that $K$ on $W$ has only two
critical values: $\max K$ and $0$. Observe also that the Hamiltonian
flow of $K$ on $U_r$ is just a reparametrization of the leafwise
geodesic flow on $M$ and the flow is the identity map outside $U_r$.

Recall that $\ai(M) = \inf \{ A(\gamma) >0 \mid \gamma \text{ is
  tangent to } \FF \text{ and contractible in } W \} > 0$, where
$A(\gamma)$ denotes the negative symplectic area bounded by the orbit
$\gamma$, and set $\ai(M) = \infty$ if there is no such $\gamma$.

The first ingredient in the proof of Theorem \ref{thm:lw} is

\begin{Lemma}
\label{lemma:K}
In the above setting, there exists a constant $C>0$, depending on
$U_r$ but not on $K$, such that whenever $\max K \geq C$, we have
$\sigma(\varphi_K) \geq C$. Moreover, $C$ can be taken to be of the
form $C=(1-O(r)) \ai(M)$. If $\ai(M)=\infty$, the constant $C$ can be
taken to be arbitrarily large.
\end{Lemma}

\begin{Remark}
  A version of this lemma also holds when $M$ is just stable although
  in this case we cannot guarantee that $C$ has the desired form.
  However, one can then take $C$ to be independent of $r$ (c.f.\
  Theorem 2.7 of \cite{Gi07coiso}).
\end{Remark}

\begin{proof}[Proof of Lemma \ref{lemma:K}]
  Set $ T_x^\perp \FF = \cap \ker {\alpha_j} \subset T_{(x,p)}
  M_p$. Then,
\begin{equation}
\labell{eq:decom}
T_{(x,p)}U = T_x^\perp \FF  \times \left( T_x \FF \times T_p B \right)
\end{equation}
is a decomposition of the tangent space at a point $(x,p) \in U$ into
a direct sum of two symplectic subspaces.  Indeed, observe that
$$
\omega_{(x,p)}= \left(1 + \sum_{j=1}^k p_j \right) {\omega_x}|_{T_x^\perp \FF}
+ \sum_{j=1}^k dp_j \wedge \alpha_j,
$$
where the first symplectic form vanishes on $T_x \FF \times T_p B$ and
the second one vanishes on $ T_x^\perp \FF $.

Let us now introduce an almost complex structure $J = J(x,p)$ on $W$,
compatible with $\omega$ and \eqref{eq:decom}. To this end, observe
that the non-degeneracy condition on $\alpha_i$'s implies that the
forms are linearly independent and leafwise closed. Let
$X_1,\ldots,X_k$ be the basis in $T_x \FF$ dual to
$\alpha_1,\ldots,\alpha_k$, i.e., $\alpha_i(X_j)=\delta_{ij}$.
Clearly, $\partial_{p_1},\ldots,\partial_{p_k}$ form a basis in $T_p
B$.  Define $J$ such that $ T_x^\perp \FF $ and $T_x \FF \times T_p B$
are complex subspaces in $T_{(x,p)}U$; furthermore, $J|_{T_x^\perp
  \FF}$ is compatible with ${\omega_x}|_{T_x^\perp \FF}$; and on $T_x
\FF \times T_p B$, we have $J(X_i)=\partial_{p_i}$. (Outside $U_R$, we
take $J$ to be an arbitrary almost complex structure compatible with
$\omega$.)  The pair $(\omega,\,J)$ gives rise to a metric on $M$
compatible with $\omega$ and such that \eqref{eq:decom} is an
orthogonal decomposition.

With respect to this metric, $\parallel
\partial_{p_i} \parallel = 1$ and $\parallel X_i \parallel =
1$. Moreover, $\parallel \omega \parallel := \sup_{X,Y}\,\omega (X,Y)
=1$, where the supremum is taken over all tangent vectors $X,\,Y$ with
norm one. Note also that the Hamiltonian vector field of
$\rho=(p_1^2+\cdots+p_k^2)/2$ is $X_\rho=\sum_{j=1}^k p_j X_j$ and,
thus, $J(X_\rho)= \sum_{j=1}^k p_j \partial_{p_j} $.

We may assume without loss of generality that the forms $\alpha_i$
extend to primitives $\bar{\alpha}_i$ of $\omega$ on $W$ such that
$$
\bar{\alpha}_i = \alpha_i + \sum_{j=1}^k p_j {\alpha}_j  \mbox{ on } U_R.
$$
In the spirit of \cite{Bo96, Bo98, Gi07coiso}, our next goal is to
define smooth one-forms $\beta_i$ on $W$ for $i=1,\ldots,k$ such that
$ \beta_i$ agrees with $\alpha_i$ on $U_r$ and with $\bar{\alpha}_i$
outside $U_R$ and that
$$
\parallel d\beta_i \parallel_{C^0} \leq 1 + O(r). 
$$
To this end, let $g=g_r(y)$ be a family of smooth, non-negative,
monotone increasing functions defined on $[0,\infty)$ such that
$g\equiv 0$ on $[0,r]$ and $g\equiv 1$ on $[R,\infty)$, and that $0
\leq g'(y) \, y + g(y) \leq 1+ O(r)$. It is not hard to write an
explicit formula for such functions $g$. Abusing notation, denote also
by $g$ the function $g(|p|)$ defined on $W$. Finally, define the
one-forms $\beta_i$ as
$$ 
\beta_i = 
\begin{cases}
\alpha_i + g(|p|) \sum_{j=1}^k p_j {\alpha}_j & \mbox{on } U_R\\
\bar{\alpha}_i & \mbox{outside } U_R
\end{cases}
$$
Now a straightforward, but tedious, calculation shows that $
\parallel d\beta_i \parallel_{C^0} \leq 1 + O(r)$. It is particularly
easy to see that this is the case when $M$ is a hypersurface. Namely,
then $k=1$ and $d\beta = (1+f) d\alpha + df \wedge \alpha$, where
$f(y)=y\,g(y)$. Notice that we only need to prove the desired estimate
on $U_R$. Then, $d\alpha = \pi^* \wM$ and the first form vanishes on
$T_x \FF \times T_p B$, while the second form vanishes on $T_x^\perp
\FF = \ker \alpha$. Thus, it suffices to prove the estimate on these
two symplectic subspaces separately. To this end, observe that on
$U_R$ we have
$$
\parallel d{\beta}|_{T_x^\perp \FF} \parallel = (1+f)/(1+p) = 1 + O(r),
$$ 
as is easy to see. On the other hand, $T_x \FF \times T_p B$ is
spanned by $\{ X, \p_p \}$ and
$$
\parallel d{\beta}|_{T_x \FF \times T_p B} \parallel = |d \beta (X, \p_p)|.
$$ 
By the definition of $g$, and, hence, of $f$, we have $d \beta (X,
\p_p) = f' \leq 1 +O(r)$. The calculation in the general case is more
involved but follows the same track as the one for hypersurfaces.

A feature of the form $\beta_i$, important in what follows, is that
\begin{equation}
\label{eq:0}
i_{X_K}d\beta_i=0.
\end{equation}
To see this, observe first that \eqref{eq:0} trivially holds outside
$U_r$, for $\supp K \subset U_r$.  Moreover, $d\beta_i=d \pi^*
\alpha_i = \pi^* \wM$ on $U_r$ and we then have
$$
i_{X_K}d\beta_i=i_{X_K}\pi^*d\alpha_i=K'(\rho)\,i_{\pi_*X_\rho}d\alpha_i=0.
$$
The last equality follows from the fact that $\pi_*X_\rho$ is tangent
to $\FF$ by Proposition~\ref{prop:flow} and that $T\FF \subset \ker
d\alpha_i$ since $M$ is stable.

Let $\epsilon >0$ be small enough so that $f=\epsilon K$ is
$C^2$-small and consider a linear homotopy from $K$ to $f$, running
through functions of $\rho$.  Then, there exist an orbit $\gamma$ of
$K$ and a homotopy trajectory $u$ partially asymptotic to $\gamma$ at
$-\infty$ and to a critical point of $f$ (and, hence, of $K$) on $M$
at $\infty$ such that
\begin{itemize}
\item[(i)] $A_K(\gamma) \leq \sigma(\varphi_K)$, and

\item[(ii)] $E(u) \leq A_K(\gamma) - \epsilon \max K \leq
  \sigma(\varphi_K) - \epsilon \max K$.
\end{itemize}
This fact is essentially a particular case of Proposition 5.4 in
\cite{Gi07coiso} with the displaceability requirement being
inessential for us and it can be proved exactly in the same way.

We claim that
\begin{equation}
\labell{eq:energy-bound}
E(u)
  \geq \parallel d\beta_i \parallel^{-1} \left|\int_{\pi(\gamma)}\alpha_i\right|.
\end{equation}
To prove \eqref{eq:energy-bound}, fix $s_n^\pm\to \pm\infty$ such that
$u(s_n^+)$ converges to 
a point of $M$ and $u(s_n^-)$ converges to $\gamma$ in
$C^\infty(S^1,W)$. Then, using \eqref{eq:0}, we have

\begin{eqnarray*}
E(u)
&=&
\int_{-\infty}^\infty \int_{S^1}
\left\| \frac{\p u}{\p s}\right\|
\cdot
\left\| \frac{\p u}{\p t}-X_K\right\|
\,dt\,ds
\\
&\geq&
\parallel d\beta_i \parallel^{-1}
\int_{-\infty}^\infty \int_{S^1}
\left|d\beta_i\left(\frac{\p u}{\p s}, \frac{\p u}{\p t}-X_K\right)\right|
\,dt\,ds
\\
&\geq&
\parallel d\beta_i \parallel^{-1}
\liminf_{n\to\infty}
\left|\int_{s_n^-}^{s_n^+}\int_{S^1}
d\beta_i\left(\frac{\p u}{\p s}, \frac{\p u}{\p t}-X_K\right)
\,dt\,ds\right|
\\
&=&
\parallel d\beta_i \parallel^{-1}
\liminf_{n\to\infty}
\left|\int_{s_n^-}^{s_n^+}\int_{S^1}
d\beta_i\left(\frac{\p u}{\p s}, \frac{\p u}{\p t}\right)
\,dt\,ds\right|.
\end{eqnarray*}
By Stokes' formula,
$$
\left|\int_{s_n^-}^{s_n^+}\int_{S^1}
d\beta_i\left(\frac{\p u}{\p s}, \frac{\p u}{\p t}\right)
\,dt\,ds\right|
=
\left|\int_{u(s_n^+)}\beta_i-\int_{u(s_n^-)}\beta_i\right|
\longrightarrow
\left|\int_{\gamma}\beta_i\right|
$$
as $n\to\infty$.
Furthermore, recall that $\gamma$ is contained in $U_r$
and $\beta_i|_{U_r}=\pi^*\alpha_i$. Thus,
$$
\left|\int_{\gamma}\beta_i\right|=\left|\int_{\gamma}\pi^*\alpha_i\right|
=\left|\int_{\pi(\gamma)}\alpha_i\right|,
$$
which completes the proof of \eqref{eq:energy-bound}.

Set $C= \parallel d\beta_i \parallel^{-1} \ai(M) $, assuming that
$\ai(M)< \infty$. It is clear that then $C =\left( 1 - O(r) \right)
\ai(M) >0$. Assume now that $\max K \geq C$ and recall that $\sigma
(\varphi_K) \leq \max K$.  If $\sigma(\varphi_K) = \max K$, the proof
is finished. Let us focus on the case when $\sigma(\varphi_K) < \max
K$. Then, $\gamma \subset U_r$ is necessarily nontrivial.  Indeed,
trivial orbits of $K$ occur where $K$ is constant, i.e. $K=0$ or $K =
\max K$. Since (i) implies that $A_K(\gamma) \leq \sigma(\varphi_K) <
\max K$, we must have $K = 0$ whenever $\gamma$ is constant and,
hence, $A_K(\gamma)=0$. Consequently, we infer from (ii) that $E(u) <
0$, which is clearly a contradiction.  Hence, $\gamma$ is nontrivial
and, as an immediate consequence, we have
$\left|\int_{\pi(\gamma)}\alpha_i\right| > 0$.  Then, by definition of
$\ai(M)$,
\begin{equation*}
\sigma(\varphi_K)
\geq
E(u)
\geq 
\parallel d\beta_i \parallel^{-1} \left|\int_{\pi(\gamma)}\alpha_i\right|
\geq
\parallel d\beta_i \parallel^{-1} \ai(M)
=
C.
\end{equation*}

Finally, note that if $\ai(M)= \infty$, we must necessarily have
$\sigma(\varphi_K) = \max K$ and, then, $C$ can be taken to be
arbitrarily large.
\end{proof}

Returning to the proof of Theorem \ref{thm:lw}, let $C >0$ be a
constant as in Lemma \ref{lemma:K} and assume from now on that the
function $K$ has the additional property that $\max K = 2 C >
C$. Thus, $\sigma(\varphi_K) \geq C$.

For the sake of convenience, let us reparametrize $H$ and $K$, as
functions of $s$ and $t$, so that $H_t \equiv 0$ on $[0,1/2]$ and $K_t
\equiv 0$ on $[1/2,1]$. (Note that the time-one maps, the Hofer norms,
the action spectra, and the action selectors remain the same.) Then,
the flow given by (a smooth reparametrization of) the concatenation of
paths $\varphi_K^t$ and $\varphi_H^t\varphi_K$, where $t \in [0,1]$,
is homotopic with fixed end points to $\varphi_H^t\varphi_K^t$, and is
generated by the Hamiltonian $H_t+K_t$.

Consider the family of diffeomorphisms
$\psi_s^t=\varphi_H^t\varphi_{sK}^t$ for $s\in [0,\,1]$, starting at
$\psi_0^t=\varphi_H^t$ and ending at
$\psi_1^t=\varphi_H^t\varphi_{K}^t$. Note that, for a fixed $s$, we
have $\psi_s ^0 = \id$ and $\psi_s^1=\psi_s=\varphi_H \varphi_{sK}$,
and, thus each $\psi_s^t$, as an an element of $\Hamt$, is generated
by the Hamiltonian
$H_t + sK_t =: G_t^{s}$. 
Examining the fixed points of $\psi_s$, as is easy to see using the
assumption that $\Fix(\varphi_H) \cap U_r = \emptyset$, we have the
disjoint union decomposition
$$
\Fix(\psi_s)=\Fix(\varphi_H)\sqcup Z_s, \text{ where } Z_s=\{x\in U_r
\mid \varphi_H^{-1}(x)=\varphi_{sK}(x)\}.
$$
Furthermore, $\sigma(\psi_s)$, as a function of $s$, is not
constant. Indeed, first note that, by our assumption on $H$,
$$\sigma(\psi_0) = \sigma(\varphi_H)\leq \|H\| < C. $$ 
Since $H\geq 0$ and $\sigma$ is monotone, we also have the inequality
$$
\sigma(\psi_1)=\sigma(\varphi_H\varphi_K) \geq \sigma(\varphi_K).
$$
Finally, using Lemma \ref{lemma:K} along with these two inequalities,
we obtain
$$
\sigma(\psi_0) < C \leq \sigma(\varphi_K) \leq \sigma(\psi_1).
$$
This shows that $\sigma(\psi_s)$ is not constant. On the other hand,
$\sigma(\psi_s)$ is a continuous function of $s$ and the set $\Fix
(\varphi_H)$ is nowhere dense. Thus, $Z_{s_0} \neq \emptyset$ for some
$s_0 \in (0,1]$ and $\sigma(\psi_{s_0})$ is the action value of
$\psi_{s_0}$ on a fixed point $x \in Z_{s_0}$. (Notice that $Z_0= \Fix
(\varphi_H) \cap U_r = \emptyset$.)

Let $\gamma(t)$ be the one-periodic orbit of the time-dependent flow
$\psi_{s_0}^t = \varphi_H^t \varphi^t_{s_0K}$, passing through
$x=\gamma(0)$. This flow is generated by $G^{s_0}_t= H_t + s_0 K_t$
and, due to the above parametrizations of $H$ and $K$, the orbit
$\gamma$ is comprised of two parts: $\gamma_K(t)=\varphi^t_{s_0K}(x)$
-- a trajectory of $s_0K$ beginning at $x$ and ending at
$y=\varphi_{s_0K}(x)$ -- and $\gamma_H(t)=\varphi_H^t(y)$ -- a
trajectory of $H$ beginning at $y$ and ending at $x$. Note that $x$
and $y$ lie on the same leaf of the characteristic foliation on some
$M_p$ and $\varphi_H(y)=x$, where $p \in B_r^k$ such that $0<|p|<r$.
 
Furthermore,
\begin{equation*}
\label{eq:gamma}
\sigma(\psi_{s_0})=A_{G^{s_0}}(\gamma) =-\int_{\gamma_K}\lambda+
\int_0^{1/2} s_0K_t(\gamma_K(t))\,dt + A_H(\gamma_H),
\end{equation*}
where $\lambda$ is a global primitive of $\omega$. Our next goal is to
show that
\begin{equation}
\label{eq:T}
\left|\int_{\gamma_K}\lambda \right|
=
\left| -\sigma(\psi_{s_0})+
\int_0^{1/2}s_0K_t(\gamma_K(t))\,dt + A_H(\gamma_H) \right|
\end{equation}
is bounded by a constant independent of $K$, which will in turn imply
that the time required to move $x$ to $y$ is (uniformly) bounded. To
see this, first recall that
$$
0 < \sigma(\psi_{s_0}) \leq \| H_t + s_0 K_t \| \leq \|H\| + \|K\|
\leq C + 2C = 3C.
$$
The second term in \eqref{eq:T} is bounded from below by zero since
$K\geq 0$ and from above by $\max K/2 = C$. Finally,
$|A_H(\gamma_H)|$ is bounded. Indeed, consider the function
$h(z)=A_H(\varphi_H^t(z))$, where $t\in [0,1]$, defined using the
primitive $\lambda$. (For example, $h(y)=A_H(\gamma_H)$.) This
function is compactly supported and independent of $K$. Letting $C' =
\max \{|\max h|,|\min h|\}$, a constant independent of $K$, we see
that $|\int_{\gamma_K}\lambda| \leq 4 C + C'$.

Let $X=X_{\rho}/|p|$ and note that $\alpha_i (X)=p_i/|p|$, where
$|p|=\sqrt{2\rho}$. Consider the primitives $\lambda = \bar{\alpha}_i$
of $\omega$ for $i=1,\ldots,k$.  As was shown above,
$$
\left| \int_{\gamma_K}\lambda \right| 
= 
\left| \int_{\gamma_K}\bar{\alpha}_i \right| 
\leq
C_i,
$$ 
where $C_i >0$ is a constant independent of $K$. On the other hand,
on $U_R$ we have
$$
\bar{\alpha}_i (X) =
\alpha_i (X) + \sum_{j=1}^k p_j \alpha_j (X) = \frac{p_i}{|p|} + |p|.
$$
As a result, 
$$
\int_{\gamma _K} \bar{\alpha}_i = 
\int_0^T \bar{\alpha}_i (X) \, dt = 
T \left( \frac{p_i}{|p|} + |p| \right), 
$$
where $T$ is the time required for the flow of $X$ to move $x$ to $y$.
For $\bar{C} = \sum_{i=1}^k C_i$, we then have
\begin{eqnarray*}
\bar{C} \geq
\sum_{i=1}^k \left| \int_{\gamma _K} \bar{\alpha}_i \right|
&=&
\sum_{i=1}^k T \left| \frac{p_i}{|p|} + |p| \right|
\\ 
&\geq& T \left( \sum_{i=1}^k \frac{|p_i|}{|p|}  -
k \, |p| \right) 
\\ 
&\geq& T \left( c - k \, |p| \right), 
\end{eqnarray*}
where $ c= \min_{|p|=1} \sum_{i=1}^k |p_i| >0 $. Note that $|p| <
r$. Thus, when $r < c/ 2k$, we have $ \bar{C} \geq T \, c /2$.  Hence
$ T \leq 2 \bar{C}/c $, where the right-hand side is independent of
$K$.

As the final step of the proof, consider a sequence $r_i \to 0$ and a
sequence of Hamiltonians $K_i$ such as $K$, supported on the (smaller
and smaller) neighborhood $U_{r_i}$ of $M$. For each $K_i$, we have a
pair of points $x_i$ and $y_i$ lying on the same leaf of the
characteristic foliation on some $M_{p(i)}$ with $|p(i)|\to 0$ and
such that $\varphi_H(y_i)=x_i$. Furthermore, the time $T_i$ required
to move $x_i$ to $y_i$ is bounded from above by a constant independent
of $K_i$ and $r_i$. Applying the Arzela--Ascoli theorem and passing if
necessary to a subsequence, we obtain points $x=\lim x_i$ and $y=\lim
y_i$ on $M$ lying on the same leaf of the characteristic foliation on
$M$ and such that $\varphi_H(y)=x$.  This completes the proof of
Theorem \ref{thm:lw}.

\end{proof}



\end{document}